\documentclass[preprint,12pt]{elsarticle}

\usepackage{amssymb}

\usepackage[utf8]{inputenc}
\usepackage{amsmath}
\usepackage{amsfonts}
\usepackage{amssymb}
\usepackage{graphicx}
\usepackage{mathabx}
\usepackage{mathtools}
\usepackage{breqn}
\usepackage{mathrsfs}
\usepackage{bm}
\usepackage{amsmath,amsfonts,amssymb,amsthm,epsfig,epstopdf,url,array}
\usepackage[font=small,labelfont=bf]{caption}%
\usepackage{hyperref}
\usepackage{subcaption}
\usepackage{tikz}
\usetikzlibrary{matrix}
\usepackage{cleveref}

\newtheorem{mydef}{Definition}[subsection]
\newtheorem{prop}{Proposition}[subsection]
\newtheorem{thm}{Theorem}[subsection]
\newtheorem{lem}{Lemma}[subsection]
\newtheorem{cor}{Corollary}[subsection]
\newtheorem{ex}{Example}[subsection]
\newtheorem{cond}{Condition}[subsection]
\newtheorem*{remark}{Remark}

\newcommand{\R}{\mathbb{R}}
\newcommand{\Q}{\mathbb{Q}}
\newcommand{\Z}{\mathbb{Z}}
\newcommand{\WH}{\widetilde{H}}
\newcommand{\booktitle}[1]{\textit{#1}}

\DeclareMathOperator*{\im}{im}
\DeclareMathOperator*{\rk}{rk}
\DeclareMathOperator*{\wt}{wt}
\DeclareMathOperator*{\supp}{supp}

\DeclareMathOperator*{\Torsion}{Torsion}

\journal{XXX}

\begin{document}

\begin{frontmatter}



\title{Winding number and Cutting number of Harmonic cycle}


\author[1]{Younng-Jin Kim}
\ead{sptz@snu.ac.kr}
\author[1]{Woong Kook}
\ead{woongkook@snu.ac.kr}

\address[1]{Department of Mathematical Sciences, Seoul National University}

\begin{abstract}
A harmonic cycle $\lambda$, also called a discrete harmonic form, is a solution of the Laplace's equation with the combinatorial Laplace operator obtained from the boundary operators of a chain complex. By the combinatorial Hodge theory, harmonic spaces are isomorphic to the homology groups with real coefficients. 
In particular, if a cell complex has a one dimensional reduced homology, it has a unique harmonic cycle up to scalar, which we call the \emph{standard harmonic cycle}.
In this paper, we will present a formula for the standard harmonic cycle $\lambda$ of a cell complex based on a high-dimensional generalization of cycletrees. 
Moreover, by using duality, we will define the standard harmonic cocycle $\lambda^*$, and show intriguing combinatorial properties of $\lambda$ and $\lambda^*$ in relation to (dual) spanning trees, (dual) cycletrees, winding numbers $w(\cdot)$ and cutting numbers $c(\cdot)$ in high dimensions.
\end{abstract}


\end{frontmatter}




\section{Introduction} \label{section_introduction}

A harmonic cycle is an element of the kernel of the combinatorial Laplacians associated with a finite chain complex. The purpose of this paper is to prove combinatorial formulas for harmonic cycles of cell complexes. In particular, this paper includes generalizations of the harmonic cycles for graphs \cite{KW}. Moreover, our formulas are constructed based on explicit combinatorial structures that reflect the \emph{duality} of a harmonic cycle as a cycle and a cocycle. 

An important motivation for studying harmonic cycles is the combinatorial Hodge theory (refer to \cite{E} and \cite{D}), which states that the harmonic space is isomorphic to homology groups with rational coefficients.  A harmonic cycle is energy minimizing among its homologous cycles, and as we shall see, allows intriguing combinatorial interpretations. For a basic material about harmonic cycles, check \cite{LIM}. See \cite{NS} and \cite{Ca2} for previous studies concerning harmonic spaces as a projection of cycle spaces. Originally, Hodge theory \cite{Ho} was developed using DeRham cohomology on a geometric manifold. Note that this methodology can be used for hole analysis in a brain network \cite{Lee}.

There are two important aspects in our current work as generalizations of our previous results \cite{KW}. One is a high-dimensional generalization of cycletrees as a basis for a formula of harmonic cycles of a cell complexes. A cycletree in a graph is also called a unicycle spanning graph or a cycle-rooted spanning tree (\cite{Be, Ke, KW}).  A high-dimensional cycletree is obtained by adding a cell to a high-dimensional spanning tree. See \Cref{section_cycletree} for details and refer to \cite{Ka} and \cite{Ca2, DKM, Ko} for high-dimensional spanning trees.

The cycletrees also play an important role in defining high-dimensional rational \emph{winding numbers} for cycles. See \Cref{section_winding}. We will define the standard harmonic cycle and discuss relations between the winding number and the standard harmonic cycle in \Cref{section_harmonic_cycle}. For the background of the winding numbers in graphs whose vertices in the plane, refer to \cite{G}.



The other aspect is \emph{dualization} and \emph{complementation} of spanning trees and cycletrees to obtain further formulas for harmonic cycles. For a dual of spanning trees, refer to \cite{Ca2, DKM}. In \Cref{section_dual_spanning_tree} and \Cref{section_dual_cycletree}, we will derive explicit relations among the resulting combinatorial objects including dual spanning trees and dual cycletrees, and introduce the notion of cutting number of cocycles including bonds which are cuts of special type. 

In \Cref{section_harmonic_cocycle}, we will present the cocycle decomposition of a harmonic cycle.
By introducing the standard harmonic cocycle, we drive a formula between cutting numbers and harmonic cycles. Finally, we define the normalized harmonic cycle from the relation between the standard harmonic cycle and cocycle. We show that the norm of the normalized harmonic cycle is the product of the number of spanning trees and the number of dual spanning trees.

\section{Preliminaries}
\subsection{Review of finite chain complex and (co)homology}
In this paper, we will assume the basic knowledge of finite chain complexes and homology groups with coefficients in $\Z$ or $\Q$ as computational tools. Familiarity with cell complexes will be helpful. For details, one may refer to standard texts in algebraic topology such as \cite{Ha, Mu} for details.

For $d>0$, let $X=X_0 \coprod \cdots \coprod X_d$ where $X_i$ is a nonempty finite set for each $i\in [0,d]$. We will refer to $X$ as a \emph{complex} of dimension $d$. (For example, $X$ may be a cell complex of dimension $d$ with $(X)_i$ the set of $i$-dimensional cells in X.) The $i$-th chain group of $X$ is the free abelian group $C_i=C_i(X)\cong\Z^{\rvert X_i \lvert}$ generated by $X_i$. The $i$-th boundary map $\partial_i=\partial_{i} (X):C_i\rightarrow C_{i-1}$ is an integer matrix whose rows and columns are indexed by $X_{i-1}$ and $X_i$, respectively, satisfying $\partial_{i-1}\partial_{i}=0$ for all $i$. We define $\partial_{i}=0$ for $i\notin [0,d]$. See below for $\partial_{0}$. The $i$-th coboundary map $\partial^t_{i}:C_{i-1}\rightarrow C_i$ is the transpose of $\partial_i$. We will often regard $X$ as the generators of the chain complex $\{C_i,\partial_i \}_{i\in [0,d]}$.

The elements of $Z_i=\ker\partial_i$ and $Z^i=\ker\partial^t_{i+1}$ are called $i$-cycles and $i$-cocycles, respectively. And the elements of $B_i=\im\partial_{i+1}$ and $B^i=\im\partial^t_{i-1}$ are called $i$-boundaries and $i$-coboundaries, respectively. The $i$-th homology and cohomology groups of $X$ are defined by $H_i(X)=Z_i / B_i$ and $H^i(X)=Z^i / B^i$. The $i$-th chain group and (co)homology group of $X$ with $\Q$-coefficients are denoted by $C_i(X;\Q)$ and $H_i(X;\Q)$ ($H^i(X;\Q)$.) As we shall see, we have  $H_i(X;\Q)\cong H^i(X;\Q)$ for all $i$ as $\Q$-vector spaces.

We have the augmented chain complex $\{C_i,\partial_i \}_{i\in [-1,d]}$ with $C_{-1}=\Z$ generated by the empty set and $\partial_{0}=\overrightarrow{1}:C_0\rightarrow C_{-1}$ defined by $\partial_{0}(x)=1$ for every $x\in X_{0}$. 
The homology of the augmented chain complex is the reduced homology $\WH_{i}(X)$. It is trivial to check that $H_0(X) = \WH_0(X)\oplus \Z$ and $\WH_i(X)=H_i(X)$ for $i\geq 1$. In this paper, we will usually work with reduced homology rather than homology.

\subsection{High dimensional spanning trees}
Let $X$ be a cell complex with the set of $i$-cells $(X)_i$ and the $i$-skeleton $X^i=(X)_{0}\cup\cdots \cup (X)_{i}$. 
Let $\mathcal{S}_i$ be a collection of subcomplexes of $X$ defined by $\mathcal{S}_i = \{Y\mid X^{i-1}\subseteq Y \subseteq X^{i}\}$. Then for any $Y\in \mathcal{S}_i$, we may write $Y=X^{i-1}\cup I(Y)$ where $I(Y)$ is a subset of  $(X)_i$ determined by $Y$.

Suppose the rank of  $\WH_{i-1}(X)$ is zero, i.e., $X$ is \emph{connected in dimension $i$}. From the definition \cite{DKM, Ka}, an $i$-dimensional spanning tree (or, $i$-tree, for short) $T$ is an element of $\mathcal{S}_i$ such that the columns of $\partial_i$ indexed by $I(T)$ form a $\mathbb{Q}$-basis of the column space $\im\partial_{i}$. Equivalently, an $i$-tree $T$ is an element of $\mathcal{S}_i$ such that $\WH_i(T)=0$ and $\WH_{i-1}(T)$ is finite. Let $\mathcal{T}_i=\mathcal{T}_{i}(X)$ denote the set of all $i$-dimensional spanning trees in $X$.

Define the weight of an $i$-tree $T$ to be $\wt(T)\equiv \lvert\WH_{i-1}(T)\rvert$. The  $i$-dimensional spanning tree number (or, $i$-tree number, for short) $k_i(X)$ is defined by 
$$k_i(X)=\sum_{T\in \mathcal{T}_{i}}\wt(T)^{2}$$
where the sum is over all $i$-trees $T\in \mathcal{T}_{i}$. Note that for an $i$-tree $T$, it is easy to check that $\lvert I(T)\rvert=\rk Z_{i-1}$. If a subcomplex $T\in \mathcal{S}_i$ satisfies $\lvert I(T)\rvert= \rk Z_{i-1}$, but $\lvert\WH_{i-1}(T)\rvert$ is infinite, then we will call $T$ a spanning tree with zero weight and define $\wt(T)=0$. 


The condition that $\WH_{i-1}(X)$ is finite is a generalization of connectedness in high dimensions. In other words, to define a high dimensional spanning forest $T$, the conditions that $\WH_{i-1}(T)$ and $\WH_{i-1}(X)$ are finite are not necessary, and the weight for $T$ is defined as $\lvert\Torsion(\WH_{i-1}(T))\rvert$. 

In this paper, we represent the $i$-tree number $k_i(X)$ as a determinant. To that end, note that the boundary group $B_{i-1}(X)=\im \partial_i$ is a subgroup of $Z_{i-1}(X)=\ker \partial_{i-1}$, both of which are free, being subgroups of the free abelian group  $C_{i-1}(X)$. Therefore, we can represent the columns of $\partial_i$ with respect to a basis $Z$ of $Z_{i-1}(X)$, which we will denote it by $[\partial_i]_Z$. Then, the weight of an $i$-tree $T$ is $\wt(T)= \lvert\det([\partial_i|_{I(T)}]_Z)\rvert$ where $\partial_i|_{I(T)}$ is the submatrix of $\partial_{i}$ consisting of the columns indexed by $I(T)$, and $k_i(X)$ is the sum of the squares of the $r\!\times\! r$ minors of $[\partial_i]_Z$ where $r=\rk Z_{i-1}(X)$. If there is no such minor, we will define $k_i(X)=0$. In other words, by using the Cauchy-Binet formula, we get the following proposition.

\begin{prop}\label{treenumber}
$$k_i(X)=\det([\partial_i]_Z[\partial_i]_Z^t).$$
\end{prop}
\begin{proof}
By the Cauchy-Binet formula, the right-hand side of the equation is the sum of $\det([\partial_i|_I]_Z)^2$  where $I$ is an $r$-subset of the indexing set for the columns of $[\partial_i]_Z$. If $\det ([\partial_i|_I]_Z)=0$, there is no spanning tree corresponding to  $I$ (or, we have a spanning tree of zero weight.) Otherwise, there is a spanning tree $T$ with its weight $\wt(T)=\lvert \WH_{i-1}(T)\rvert=\lvert\det([\partial_i|_I]_Z)\rvert$ where $I=I(T)$.
\end{proof}

\subsection{Harmonic space and combinatorial Hodge theory} 
In this subsection, we define harmonic cycles, a main object of study of this paper, via combinatorial Laplacians, and recall combinatorial Hodge theory relating harmonic cycles and homology groups. For readers who search for a basic material introducing harmonic spaces, refer to \cite{LIM}. For an application of this theory to the ranking system of items, can be found in \cite{JX}.

Given a cell complex $X$, the $i$-th combinatorial Laplacian $\Delta_{i}:C_{i}(X;\mathbb{Q})\rightarrow C_{i}(X;\mathbb{Q})$ is given by
$$\Delta_{i}=\partial_{i}^t \partial_{i} + \partial_{i+1} \partial^t_{i+1}$$
for each dimension $i$. The $i$-th \emph{harmonic space} $\mathscr{H}_i(X)$ of $X$ is defined to be 
$$\mathscr{H}_i(X)=\ker \Delta_{i}$$
and its elements are called $i$-\emph{harmonic cycles}.

Regard $C_{i}(X;\mathbb{Q})$ as a 
$\mathbb{Q}$-vector space endowed with a standard inner product $\circ$ so that the set $(X)_i$ of its generators forms an orthonormal basis. From the orthogonal decomposition (refer to \cite{Fr})
\begin{equation}\label{decompostion}
C_{i}(X;\mathbb{Q})=\mathscr{H}_{i}(X)\oplus\im\partial_{i}^{t}\oplus\im\partial_{i+1}
\end{equation}
one can deduce 
\begin{equation}\label{harmonics}
\mathscr{H}_i(X)=\ker\partial_{i} \cap \ker\partial^t_{i+1}
\end{equation}
via $(\im M^{t})^{\perp}=\ker M$ for a matrix $M$.  Hence, we have an important fact that \emph{a harmonic cycle is both a cycle and a cocycle}.

From the above orthogonal decomposition for $C_{i}(X;\mathbb{Q})$, we also see that $\ker\partial_{i}=\mathscr{H}_{i}(X)\oplus\im\partial_{i+1}$. Hence we have the following main statement of the combinatorial Hodge theory:  for each dimension $i$
\begin{equation}\label{hodge}
\mathscr{H}_i(X)\cong \WH_{i}(X;\mathbb{Q})
\end{equation}
as $\mathbb{Q}$-vector spaces. In particular $\rk\mathscr{H}_{i}(X)=\rk \WH_{i}(X;\mathbb{Q})$ for all $i$. (One can show that this isomorphism maps a harmonic cycle $h$ to its homology class $\bar{h}$.) 

The following \emph{energy minimizing property} of a harmonic cycle is a consequence of \Cref{hodge}: For $h\in \mathscr{H}_{i}(X)$ and $x\in \bar{h}$, 
\begin{equation}\label{energyminimizing}
h\circ h\leq x\circ x.
\end{equation} 
Indeed, this inequality follows easily from the facts $x=h+\partial_{i+1}y$ for some $y\in C_{i+1}(X;\mathbb{Q})$ and $h\circ (\partial_{i+1}y)=(\partial_{i+1}^{t}h)\circ y=0$ because $h\in\ker\partial^t_{i+1}$.

\section{Cycletree and its minimal cycle}\label{section_cycletree}
\subsection{Cycletree}
A \emph{cycletree} in a graph is a connected spanning subgraph such that the number of edges equals the number of vertices. Equivalently, a cycletree is the union of a spanning tree and an edge not in the spanning tree. Thus, a cycletree contains a unique cycle. In the literature, a cycletree is also called a unicycle graph, or a cycle rooted spanning tree. This object may be called a co-tree \cite{Ca}, or a relative acyclic complex to a given cycle. In our previous work \cite{KW}, we defined cycletrees in dimension 1. Refer to \cite{KW} for examples.

In the case of the complete graphs, the number of cycletrees is known. Refer to A057500 in On-Line Encyclopedia of Integer Sequences (OEIS, \cite{Sl}). An \emph{edge-rooted} cycletree is a cycletree together with an edge on its unique cycle. 

\begin{thm}
The number of cycletrees in a complete graph $K_n$ is $$\binom{n-1}{2}\, e^n\, \Gamma(n-2,n)$$
where $\Gamma$ is the incomplete gamma function, i.e., $\Gamma(s,x)=\int^\infty_x t^{s-1}e^{-t} \, \mathrm{d}t$.
\end{thm}
\begin{proof}
The main bottleneck of the proof is done by \cite{Be}. In that paper, there is an enumerative formula $e^n\, \Gamma(n-2,n)$ for the number of spanning trees in the complete graph with labelled vertices, 1, 2, $\cdots$, $n$, such that a spanning tree is vertex-rooted at 1, and 2 is a descendent of 3. There is a bijection between the spanning tree, and an edge-rooted cycletree whose rooted edge connects 2 and the ancestor of 3 with respect the vertex root 1 . To be specific, the bijection is an operation which adds or deletes an edge between 2 and the ancestor of 3. There are $\binom{n-1}{2}$ ways to choose vertices 2 and 3 in the complete graph. Refer A057500 in OEIS, \cite{Sl} for details.
\end{proof}

It is unknown whether there is a polynomial time algorithm to calculate the number of cycletrees, in general. However, some weighted algorithm can be devised. Let $l_j$ be the number of cycletrees each with a cycle of length $j$. 

\begin{prop}
$$(m-n+1) k_1(G)= \sum^n_{j=1}l_j j$$
where $m$ is the number of edges in $G$.
\end{prop}
\begin{proof}
There is a bijection from the set of all edge-rooted  cycletrees to the set of all pairs each consisting of a spanning tree and an edge not in the tree.
\end{proof}

By THEOREM 3 from \cite{Ke}, we get the following theorem for a cycletree-number with the weight of a squared cycle-length.

\begin{thm}
$$\det \triangle^{down}_1(t) = \sum^n_{j=1}l_j j^2t^2 + O(t^4)$$
where $t \in \R^* = \R-\{0\}$ and $\triangle^{down}_1(t)$ equals $\partial_1^t \partial_1$ except that a nonzero upper diagonal component is $-e^{it}$ and a nonzero lower diagonal component is $-e^{-it}$.
Therefore, the limit $\lim\limits_{t \to 0}\dfrac{\det \triangle^{down}_1(t)}{t^2}$ equals $\sum\limits^n_{j=1}l_j j^2$.
\end{thm}

Recall that a high dimensional spanning tree $T\in \mathcal{S}_i$ can be described as a collection of column indices of $\partial_i$ corresponding to a basis for the column space of $\partial_i$. Likewise, we can define a high dimensional cycletree $Y\in \mathcal{S}_i$ as a collection of column indices having exactly one more element than that of a spanning tree $T$. Let $\mathcal{U}_i=\mathcal{U}_i(X)$ denote the set of all $i$-dimensional cycletrees of $X$. 



Now, let us define a cycletree $U$ and its weight in detail, which will be similar to those of a spanning tree.

\begin{mydef}\label{ctweight}
Let $X$ be a cell complex with $\rk\WH_{i-1}(X)=0$ for $i\geq 0$. An $i$-dimensional cycletree ($i$-cycletree) $U$ is an element in $\mathcal{S}_i$ such that $\lvert I(U)\rvert =\rk Z_{i-1}(X)+1$. The weight of $U$ is defined to be
$$\wt(U)=\lvert\WH_{i-1}(U)\rvert$$
when $\lvert\WH_{i-1}(U)\rvert$ is finite. Otherwise, define $\wt(U) = 0$.

\end{mydef}

Suppose $U$ is an $i$-cycletree with nonzero weight. The corresponding columns in the matrix $\partial_i$ to a cycletree is linearly dependent, and there is a unique non-trivial cycle $C_U$ in $Z_{i}(X)$ up to scalar multiplication, supported by the cycletree $U$. In order to define $C_U$ uniquely, we will construct the cycle part of a cycletree $C_U$ systemically after \Cref{cyclepart}.


\subsection{Minimal cycle}
In this subsection, we will fix the cycle part $C_U$ of a cycletree $U$ and regard it as a vector
We will show that $C_U$ is indeed a minimal cycle, which will be made clear subsequently. To get the cycle part $C_U$, we need the following lemma. 

\begin{lem}[Orthogonal complement]\label{cyclepart}
For a matrix $A$ of size $n\!\times\!(n+1)$, let $[A]_i$ denote $A$ minus the $i$-th column. Define a (row) vector $v$ whose components are given by $v_i=(-1)^i\cdot \det([A]_i)$. If $A$ has rank $n$, then we have the full rank $(n+1)\!\times\!(n+1)$ matrix $M$ which consists of $A$ and $v$ as the last row. If the rank of $A$ is less than $n$, then $v$ is the zero vector. Moreover, $v^t\in \ker A$.
\end{lem}

\begin{proof} 
Let $A$ be an $n\!\times\!(n+1)$ matrix. Then, $\det(M)=(-1)^{n+1} \sum_i \det([A]_i)^2$. Therefore, if $\rk A$ is $n$, then there is a nonzero summand $\det([A]_i)$. However, if $\rk A$ is less then $n-1$, all summands $\det([A]_i)$ are zero, and $v=0$. To show the last statement, let $A_j$ is the $j$-th row of $A$, and note that for all $j$, the inner product $v \circ A_j^t=(-1)^{(n+1)}\det (A^t| A_j^t) =0$.
\end{proof}

We will let the vector $v$ in \Cref{cyclepart} be denoted by $v_A$ when necessary. In general, let $A$ be an $n\!\times\!m$ matrix with its rank $n(< m)$. With a column indexing set $I$ representing a cycletree, $A|_{I}$ is an $n\!\times\! (n+1)$ submatrix. Then, by applying \Cref{cyclepart}, we have a vector $v_{A|_{I}}$ of length $n+1$. Finally, we will extend this vector $v_{A|_{I}}$ to have length $m$ by adding 0 for the components that are not indexed by $I$. Check that $A\circ v_{A|_{I}}=0$.

By letting $A=[\partial_i]_Z$ where $Z$ is a basis of $\ker\partial_{i-1}$, we see that $v_{A|_{I(U)}}$ is the desired cycle $C_U$ for each cycletree $U$. It is worth noting that the components of the cycle $C_U$ in a cycletree $U$ are the weights of spanning trees included in $U$. Note that if a cycletree $U$ has zero weight, then  $C_U=0$ by this construction.




For a cycletree $U$, define $\gcd(C_U)$ to be the gcd of all components of the integer vector $C_U$. Note that $\gcd(C_U)$ is invariant either as $C_U\in C_i(X)$ or $C_U\in Z_i(X)$. Denote $\widehat{C_U}$ to be the integer vector $\dfrac{1}{\gcd(C_U)}C_U$. In the following proposition, we will see that the definition of $\wt(U)$ makes sense.

\begin{prop}\label{prop_gcd_wt}
For any cycletree $U$, we have $$\gcd(C_U)=\lvert\WH_{i-1}(U)\rvert.$$
\end{prop}
\begin{proof}
It is enough to show that $\gcd(v_A)=\lvert\Z^n/\im A^t\rvert$ in \Cref{cyclepart}. If $\rk A < n$, we have $\gcd(v_A)=\lvert\Z^n/\im A^t\rvert$, and they are infinite. Therefore, suppose $\rk A = n$. Note that $\gcd(v_A)$ and $\Z^n/\im A^t$ are invariant under a columnwise Gauss elimination over $\Z$ for $A$. Hence, it suffices to consider the case where $A$ is the Smith normal form without pivoting via the Gauss elimination. Then, the nonzero entries of $A$ are $d_1, \cdots, d_n$. We can check that $\gcd(v_A)=d_1\times\cdots\times d_n=\lvert\Z^n/\im A^t\rvert$.
\end{proof}
\begin{cor}
For any cycletree $U$, $$C_U = \wt(U) \widehat{C_U}.$$
\end{cor}

To define a minimal cycle, we introduce a \emph{minimally supported vector} (or, minimal vector) in a vector space. Assume $V$ is a subspace of a vector space $\Q^n$. For a vector $v\in V$, the \emph{support} of $v$ is $\supp(v)=\{i\mid v_i\neq 0 \}$. Hence the set of supports of vectors in $V$ can be regarded as a subposet of the power set $2^{[n]}$. Since the support of a vector is invariant up to nonzero scalar multiplication, we can introduce a poset structure on $V/\Q^*$ similarly, where $\Q^*=\Q-\{0\}$. A minimal vector of $V$ is a vector in an equivalence class which is a minimal element in $V/\Q^*$.
\begin{mydef}\label{minimal}
An $i$-minimal cycle is a minimal vector in the cycle space $Z_i$. Define an $i$-minimal cocycle and an $i$-minimal (co)boundary, similarly.
\end{mydef}
Since a minimal cycle is a minimal vector in a kernel, we get the following.
\begin{prop}\label{minimal2}
The support of an $i$-minimal cycle corresponds to a collection of columns of $\partial_i$ which forms a minimal linearly dependent set.
\end{prop}
\begin{proof}
Clear from the definitions.
\end{proof}

Now, we characterize the unique cycle in a cycletree as a minimal cycle.
\begin{prop}\label{prop_mincycle1}
The cycle part of an $i$-cycletree with nonzero weight is an $i$-minimal cycle, and vice versa.
\end{prop}
\begin{proof}
The cycle part $C_U$ of a cycletree $U$ is the unique solution of a linear system up to nonzero scalar multiplication. Thus, we cannot find the nonzero cycle in $C_U$ with a smaller support, which means that $C_U$ is a minimal cycle.

For the converse, suppose we have a minimal cycle $C$. Then, delete a cell $e$ in $C$. One can construct a spanning tree $T$ including $C-\{e\}$. By the definition of spanning tree, we have $e\notin T$, but $C-\{e\}\subseteq T$. Finally, we have a cycletree $U=T\coprod \{e\}$ with its cycle part $C$.
\end{proof}

\begin{prop}
The set of $i$-minimal cycles generates the cycle space $Z_i(X)$.
\end{prop}
\begin{proof}
The set of minimal cycles includes the fundamental cycles of a spanning tree.
\end{proof}

\begin{prop}\label{prop_mincycle2}
A minimal boundary in $B_i(X)$ is either a minimal cycle in $Z_i(X)$ or the sum of two minimal cycles if $\rk \WH_i(X)=1$.
\end{prop}
\begin{proof}
First, we have the orthogonal decomposition $Z_i(X;\Q)=B_i(X;\Q)\oplus \mathscr{H}_{i}(X)$ from \Cref{decompostion}. We can choose a harmonic cycle $\lambda$ such that $\lambda$ is an integer vector and $\mathscr{H}_{i}(X)=\lambda\Q$ by \Cref{hc}. Therefore, the boundary group $B_i=B_i(X;\Z)$ can be represented by a kernel, i.e., we have $B_i=\ker\overline{\partial}_i$ where $\overline{\partial}_i=\begin{pmatrix} \partial_i \\ \lambda \end{pmatrix}$ and $\lambda$ is a row vector. Now, we can treat a minimal boundary as a minimal cycle. Let $v$ be a minimal boundary in $B_i$. Then, $v$ is a minimal cycle in $\ker\overline{\partial}_i$. We have one less condition in the linear system $\overline{\partial}_i x=0$ than those in the system $\partial_i x=0$. The solution space of $\overline{\partial}_i x=0$ might be larger by at most rank 1. Therefore, this proposition holds.
\end{proof}

\begin{ex}[Minimal boundary and minimal cycle]
Figure 1-A shows a cell complex $X$. In Figure 1-B, we have the red faces, whose boundary is a minimal cycle. In Figure 1-C, we have the red faces, whose boundary is the sum of minimal cycles, the outer cycle and the inner cycle.

\begin{figure}[h]
\centering
\includegraphics[angle=0, scale=0.4]{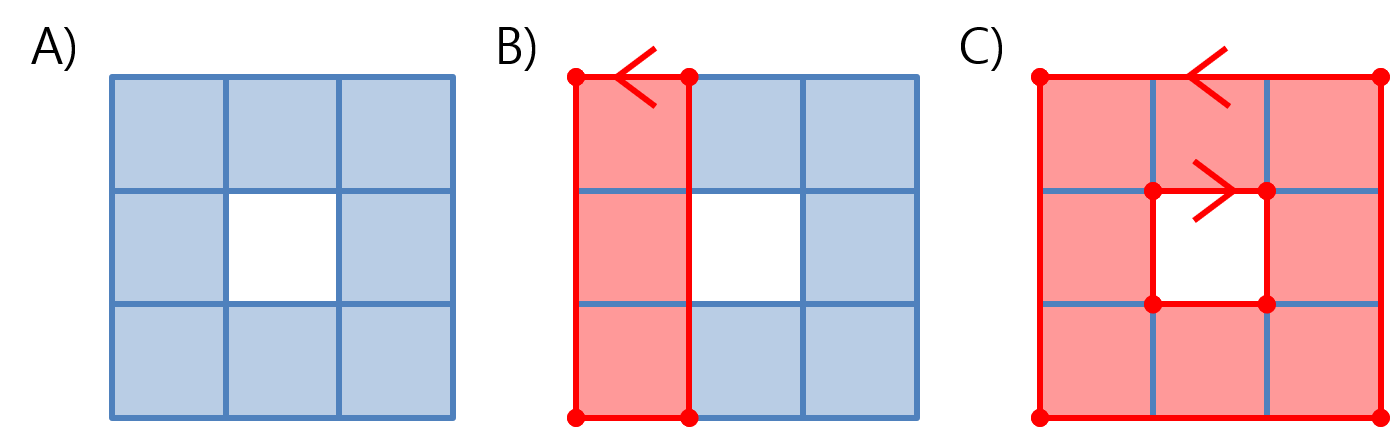}
\caption{A cell complex $X$ and minimal cycles with minimal boundaries} \label{Figure1}
\end{figure}
\end{ex}

In what follows, we will often need to \emph{annihilate} a cycle $z\in Z_{i-1}(X)$ to simplify homology. This task will be achieved purely algebraically by modifying the chain complex of $X$ without any reference to topology. Specifically, let $z$ be a cycle in $Z_{i-1}(X)$ that is not a boundary. Let $e\Z$ denote a rank 1 free abelian group generated by an element $e$, which we may call a \emph{virtual $i$-cell}. Let $X\oplus e$ denote the chain complex whose $i$-th chain group is $C_i(X)\oplus e\Z$ and $j$-th chain group is $C_j(X)$ if $j\neq i$. We define its $i$-th boundary operator to be $\partial_{X\oplus e,i}=(\partial_{X,i},\,z)$ obtained by adding $z$ to $\partial_{X,i}$ as the last column. Also, we define $\partial_{X\oplus e,j}=\partial_{X,j}$ for $j\not=i$.  Note that $\partial_{X\oplus e,i}(e)=z$, and, therefore, $z$ is annihilated in the homology of $X\oplus e$. We will refer to $e$ as \emph{a virtual $i$-cell whose boundary equals $z$}.



\begin{prop}\label{ctgcd}
Let $U$ be an $i$-cycletree with nonezero weight in a cell complex $X$ such that $\rk\WH_{i-1}(X)=0$, and $e$ be a virtual $(i+1)$-cell whose boundary is $m\widehat{C_U}$ in $C_i(X)$ where $m$ is a positive integer. Then, we have
$$\lvert\WH_i (U\oplus e)\rvert= m.$$
\end{prop}

\begin{proof}
We have $\WH_i (U\oplus e)=\cfrac{\ker \partial_{U\oplus e,i}}{\im \partial_{U\oplus e,i+1}}=\cfrac{\ker \partial_{U,i}}{\partial_{i+1}(e)}=\cfrac{\Z}{m \Z}$. By \Cref{cyclepart}, we know that $C_U$ is a non-trivial element in $\ker\partial_{U,i}$ which is is isomorphic to $\Z$, and the integer vector $\widehat{C}_U$ is a generator of $\ker \partial_{U,i}$.
\end{proof}


\section{Winding number}\label{section_winding}


The $i$-winding number $w(\cdot)$ is defined as a function on the cycle space $Z_i(X)$ (or, $\WH_i(X)$ and $\mathscr{H}_i(X)$). It measures how many times a cycle winds around a generator of homology or harmonic space. For simplicity, we will focus on the case $\rk \WH_i(X)=1$. Let $r_i$ be $\rk Z_i(X)$. 


For any $n\!\times\! m$ matrix $M$, we can find an $n\!\times\!(\rk M)$ matrix $\overline{M}$ satisfying $\im M= \im\overline{M}$ as vector spaces, for example, by choosing the columns of $\overline{M}$ to be a basis for the column space of $M$. Therefore, from $\partial_{i+1}$, we can obtain $\overline{\partial}_{i+1}$ of size $\rk C_i\!\times\!\rk\partial_{i+1}$.  Note that we have $\rk\partial_{i+1}=r_i-1$ since we assume $\rk \WH_i(X)=1$. Now, fix a basis $Z$ of $Z_i(X)$. Then, we get the $r_i \!\times\! (r_i-1)$ matrix $[\overline{\partial}_{i+1}]_Z$ by writing the columns of $\overline{\partial}_{i+1}$ with respect to $Z$.

\begin{mydef}
An $i$-winding number $w:Z_i(X)\rightarrow \mathbb{Z}$ is given by $w(z)=\det([z]_Z,[\overline{\partial}_{i+1}]_Z)$ where $Z$ is a basis of $Z_i(X)$.
\end{mydef}

Note that this construction is independent of the choice of $Z$ up to sign. The winding number can be interpreted via homology by the following proposition. 

\begin{prop} \label{winding0}
Let $e$ be a virtual $(i+1)$-cell whose boundary is a given cycle $z\in\ker\partial_{i}$. If $w(z)$ is nonzero, 
$$\lvert w(z)\rvert= \lvert\WH_i(X\oplus e)\rvert.$$
Moreover, $w(z) = 0$ if and only if $\lvert \WH_i(X\oplus e)\rvert=\infty$
\end{prop}
\begin{proof}
\begin{eqnarray*}
\lvert w(z)\rvert
&=& \lvert\det([z]_Z,[\overline{\partial}_{X,i+1}]_Z)\rvert \\
&=& \lvert\ker\partial_{i}/(\im\overline{\partial}_{X,i+1}+z\Z)\rvert \\
&=& \lvert\ker\partial_{i}/(\im\partial_{X,i+1}+\partial_{i+1}(e)\Z)\rvert \\
&=& \lvert\WH_i(X\oplus e)\rvert
\end{eqnarray*}
Moreover, $w(z) = 0$ means $z\in \im \partial_{X,i+1}$. Equivalently, $\WH_i(X\oplus e)=\WH_i(X)$ has rank 1.
\end{proof}

Since we represent $\lvert w(C_U) \rvert$ in terms of homology, we can analyze it via a long exact homology sequence.

\begin{prop} \label{winding}
Let $U$ be an $i$-cycletree with nonezero weight in a cell complex $X$. If $w(C_U)\neq 0$, we have
$$\wt(U)\lvert w(\widehat{C_U})\rvert=\lvert H_{i}(X,U)\rvert\lvert\WH_{i-1}(X)\rvert.$$
Moreover, $w(C_U) = 0$ if and only if $\lvert H_{i}(X,U)\rvert$ is infinite.
\end{prop}

\begin{proof}
Let $e$ be a virtual $(i+1)$-cell with $\partial_{i+1}(e)=\widehat{C_U}$. Then, from \Cref{ctgcd}, we get 
$\lvert\WH_i(U\oplus e)\rvert= 1$ (i.e., $\WH_i(U\oplus e)=0$). One can easily check that $H_{i-1}(X\oplus e,U\oplus e)=H_{i-1}(X,U)$. Moreover, we know $H_{i-1}(X,U)=0$.




Now, consider the long exact homology sequence of the pair $(X\oplus e, U\oplus e)$:
$$0 \to \WH_i(X\oplus e)\to H_i(X\oplus e,U\oplus e)\to \WH_{i-1}(U\oplus e) \to \WH_{i-1}(X\oplus e) \to 0.$$


Again, we have $H_i(X\oplus e,U\oplus e)=H_i(X,U)$. Also, one can easily check that $\WH_{i-1}(U\oplus e)=\WH_{i-1}(U)$ and $\WH_{i-1}(X\oplus e)=\WH_{i-1}(X)$. Therefore, the exact sequence is $0\to \WH_i(X\oplus e)\to H_i(X,U)\to \WH_{i-1}(U) \to \WH_{i-1}(X) \to 0$ from which we obtain $$\lvert H_i(X,U)\rvert\lvert\WH_{i-1}(X)\rvert = \lvert\WH_i(X\oplus e)\rvert\lvert\WH_{i-1}(U)\rvert.$$

If $w(C_U)\neq 0$, then we have 
$\lvert\WH_i(X\oplus e)\rvert = \lvert w(\widehat{C_U})\rvert$ by \Cref{winding0}, and $\lvert\WH_{i-1}(U)\rvert=\wt(U)$.
Otherwise, by \Cref{winding0} again, we can show that $w(C_U) = 0$ is equivalent to $\lvert H_{i}(X,U)\rvert = \infty$.
\end{proof}

\begin{ex}[Intuition for the winding number]\label{ex_winding}
In Figure 2-A shows a harmonic cycle $\lambda$ of a cell complex $X$. The relation between a harmonic cycle and the winding number will be discussed in the following section. In Figure 2-B and 2-C, we have cycles $v_B$ and $v_C$ marked red. The winding number for $v_B$ gives $w(v_B)=\pm 1$ where the sign depends on the initial setting to construct $w(\cdot)$. The winding number for $v_C$ is $w(v_C)=0$.

\begin{figure}[h]
\centering
\includegraphics[angle=0, scale=0.5]{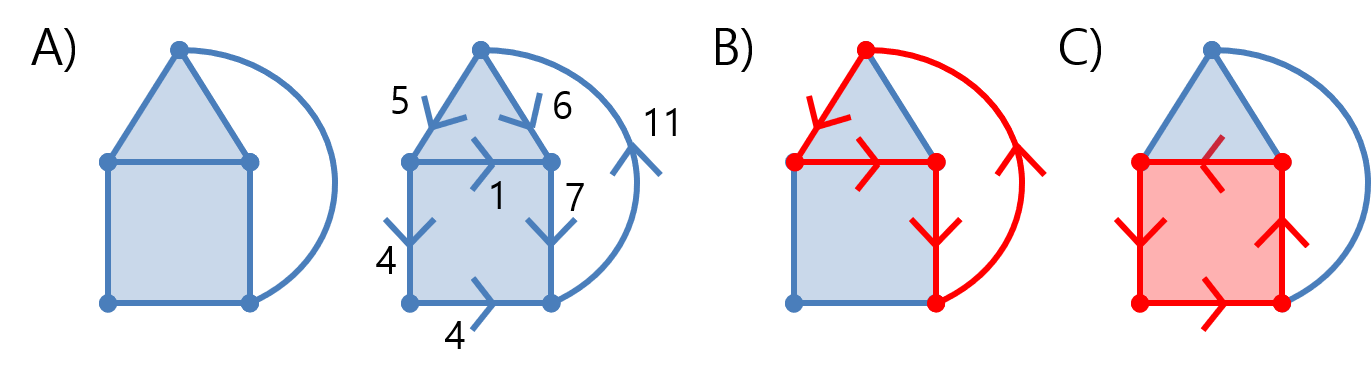}
\caption{A cell complex $X$, a harmonic cycle, and cycles} \label{Figure2}
\end{figure}

\end{ex}

\section{Standard harmonic cycle}\label{section_harmonic_cycle}

In this section, we define the standard harmonic cycle $\lambda$ and establish its formula via the winding number. 

\begin{mydef}\label{hc} Let $X$ be a cell complex with $\rk\WH_{i-1}(X)=0$ and $\rk\WH_i(X)=1$. Then the $i$-th standard harmonic cycle $\lambda$ in $Z_i(X)$ is defined to be
$$\lambda = \sum_U w(C_U) C_U $$
where the summation is over all cycletrees $U\in \mathcal{U}_i=\mathcal{U}_i(X)$.
\end{mydef}
Throughout this section, assume that $X$ is a cell complex with $\rk\WH_{i-1}(X)=0$ and $\rk\WH_i(X)=1$.
The above expression for $\lambda$ may be deduced from \cite[Theorem A]{Ca2} with some scalar multiplication.
Now, we describe the relation between the standard harmonic cycle and the winding number.

\begin{thm}\label{inner}
(Inner product formula)
For any $z\in Z_i(X)$, we have 
$$z\circ \lambda  = w(z) k_i(X)$$
where $\circ$ is the inner product in $C_i(X)$.
\end{thm}
\begin{proof}
It's enough to show the following equalities step by step:
\begin{eqnarray*}
w(z) k_i(X)&=&\sum_{\lvert I\rvert= n} w(z) \det([\partial_i|_I]_Z)^2\\
	&=&\sum_{\lvert I\rvert= n}\ \sum_{I(U)=I\coprod e} 
	\frac{z_e}{C_{U,e}} w(C_U) \det([\partial_i|_I]_Z)^2\\
	&=&\sum_{\lvert I\rvert= n}\ \sum_{I(U)=I\coprod e} 
	z_e C_{U,e} w(C_U)\\
	&=&\sum_{U\in \mathcal{U}_i}\ \sum_{e \in U} 
	z_e C_{U,e} w(C_U) \\
	&=& \sum_{U\in \mathcal{U}_i}(z\circ C_{U}) w(C_{U})\, 
\end{eqnarray*}

The first equality comes from the definition of the weighted tree-number and the proof of \Cref{treenumber}. We have $z=\sum_e \frac{z_e}{C_{U,e}}C_U$ because one can find a basis $\{C_U\mid I(U)=I\coprod e\}$ of $Z_i(X;\Q)$ where $C_{U}$ represents the unique cycle in the union of the spanning trees corresponding to the indexing set $I$ and an edge $e$. Therefore, the second equality holds. The third equality follows because $C_{U,e}=\pm \det([\partial_i|_I]_Z)$ from the construction of $C_U$. By reindexing the edge-rooted cycletrees, we have the fourth equality. The last equality follows from the definition of  inner product.
\end{proof}

\begin{cor}\label{inner_cor}
$\lambda$ is a nonzero element in the harmonic space $\mathcal{H}_i(X).$
\end{cor}
\begin{proof}
By setting $z=\lambda$ and applying \Cref{inner}, we have
$$\lambda \circ \lambda  = w(\lambda) k_i(X)= k_i(X)\sum_{U\in \mathcal{U}_i} w(C_U)^2.$$ Since the $i$-th homology is nontrivial, there is at least one cycletree $U$ with $w(C_U)\neq 0$, and a spanning tree. Hence the right-hand side of the second equality is nonzero, and we conclude $\lambda\neq 0$.  
\end{proof}

Now, we have an enumeration formula for the cycletrees as in \Cref{section_cycletree}.

\begin{cor}
Let $X$ be a cell complex satisfying $\rk\WH_{i-1}(X)=0$ and $\rk\WH_i(X)=1$. Then, we have
$$\frac{\lambda \circ \lambda}{k_i(X)} = \sum_{U\in \mathcal{U}_i} w(C_U)^2.$$ 
\end{cor}

Note that the left-hand side can be calculated in polynomial time, i.e., $\lambda$ can be determined using \Cref{inner}.

\begin{ex}[Inner product formula]
In \Cref{ex_winding}, the given harmonic cycle is actually the standard harmonic cycle. Thus, we have the standard harmonic cycle $\lambda$ of a cell complex $X$, and cycles $v_B$ and $v_C$. The left-hand side of the inner product formula for $v_B$ is $v_B\circ \lambda =11+5+1+7=24$. And the right-hand side is $w(v_B) k_1(X) = 1\times 24$ where $k_1(X) = 24$. Moreover, the inner product formula for $v_C$ holds since $v_C\circ \lambda =-7-1+4+4=0$ and $w(v_C) = 0$.
\end{ex}

\begin{remark}
From \Cref{inner}, we can extend the domain of the winding number map from a cycle space to a chain group, i.e., $w(z) \coloneqq z\circ \lambda / k_i(X)$ where $z\in C_i(X)$. We will call it the rational winding number.
\end{remark}

\section{Duality and dual spanning tree}\label{section_dual_spanning_tree}
In this section, we define a dual spanning tree of a cell complex $X$ and discuss properties in \Cref{condition}. To that end, we first define a complement operator. Note that, to analyze a certain hole, we focus on the case when \Cref{condition} holds. If we have multiple holes, we can deal each holes separately as in \cite{KW}.




\begin{mydef}\label{def_complement}
Let $X$ be a cell complex. Recall that $\mathcal{S}_i$ is a collection of subcomplexes $Y$ of $X$ such that $Y=X^{i-1}\cup I(Y)$ for a subset $I(Y)\subseteq (X)_i$. The complement operator $\overline{\ \cdot\ }$ in dimension $i$ is a bijection defined on $\mathcal{S}_i$ by $\overline{Y}=X^{i-1} \cup I(\overline{Y})$ where $(X)_i=I(\overline{Y}) \coprod I(Y)$, or equivalently, $(\overline{Y})_i=\overline{(Y)_i}$.
\end{mydef}

It is trivial to show $Y=\overline{\overline{Y}}$. We will specify the dimension for a complement operator if necessary. Now, we introduce the dual spanning tree $T^*$ in a cell complex $X$, which is the conceptual dual to the spanning tree.

Recall that $\rk\WH_{i-1}(X)=0$ means connectedness in dimension $i$. Similarly, we need $\rk\WH_{i+1}(X)=0$ for the dual concept of a spanning tree.
Assume $X$ is the cell complex with $\rk\WH_{i+1}(X)=0$ throughout this section.

\begin{mydef}
An $i$-dimensional dual spanning tree $T^*$ is an element of $\mathcal{S}_i$ with $I=I(T^*)$ satisfying
$\lvert I\rvert=\rk Z^{i+1}(X)$. We define the weight of $T^*$ to be $\wt(T^*)=\lvert\det([\delta_{i}|_I]_Z)\rvert$ where $Z$ is a basis of $Z^{i+1}(X)$.
\end{mydef}
In this paper, $*$  is not an operator but a symbol for a dual object. This dual version of a spanning tree is related to the concepts in \cite{Ca2,DKM}.


Now, let $\mathcal{T}^i=\mathcal{T}^i(X)$ be the set of all $i$-dimensional dual spanning trees in $X$. We define the $i$-th dual tree number $k^i(X)$ by 
$$k^i(X)=\det([\delta_{i}]_Z [\delta_{i}]^t_Z)$$
using the Cauchy-Binet formula as in \Cref{treenumber}. In the following lemma, $\delta_{(X,A),i}$ where $A$ is a subcomplex of $X$ denotes a coboundary operator for the \emph{relative} cochain complex $\{C^{i}(X,A)=C^{i}(X)/C^{i}(A)\}$ induced by $\delta_{X,i}$.

\begin{lem}\label{prop1}
$\delta_{(X,\overline{T^*}),i} = \delta_{X,i}|_I$ where the column indexing set $I$ represents a dual spanning tree $T^*$.
\end{lem}
\begin{proof}
This is because $C^{i}(X)$, $C^i(T^*)$, and $C^i(\overline{T^*})$ are free $\Z$-modules indexed by the $i$-cells of $X$, $T^*$ and $\overline{T^*}$, respectively. Thus, we have $C^{i}(X)=C^i(T^*)\oplus C^i(\overline{T^*})$ and $C^{i}(X,\overline{T^*})\cong C^i(T^*)$.
\end{proof}

The weight for a dual spanning tree $T^*$ can be written via (co)homology in a similar way as that of a spanning tree $T$ using the following lemma.
\begin{lem}\label{lem_prop2} $\lvert H^{i+1}(X,\overline{T^*})\rvert=\lvert H_{i}(X,\overline{T^*})\rvert$ for a dual spanning tree $T^*$.
\end{lem}
\begin{proof}
Let $I=I(T^*)$. The assumption $\rk \WH^{i+1}(X)=0$ implies $\lvert I\rvert=\rk Z^{i+1}(X)=\rk B^{i+1}(X)$. 
Since $\rk C_{i+1}(X)=\rk C_{i+1}(X,\overline{T^*})$, $\rk B_{i+1}(X)=\rk B_{i+1}(X,\overline{T^*})$, and $B^{i}(X,\overline{T^*})=0$ from \Cref{decompostion}, we know
\begin{equation}\label{eq_1}
\rk \WH_{i+1}(X)+\rk B^{i+1}(X) =\rk H_{i+1}(X,\overline{T^*})+\rk B^{i+1}(X,\overline{T^*})
\end{equation}
\begin{equation}\label{eq_2}
\lvert I\rvert = \rk C_{i}(X,\overline{T^*})=\rk B_{i}(X,\overline{T^*})+\rk H_{i}(X,\overline{T^*})
\end{equation}
By \Cref{eq_1,eq_2}, we have $\rk \WH^{i+1}(X,\overline{T^*})=\rk \WH_{i}(X,\overline{T^*})$.

Furthermore, we know $\Torsion(\WH^{i+1}(X,\overline{T^*})) = \Torsion(\WH_{i}(X,\overline{T^*}))$ by the universal coefficient theorem. Thus, we have $\lvert \WH^{i+1}(X,\overline{T^*})\rvert=\lvert \WH_{i}(X,\overline{T^*})\rvert$.
\end{proof}

\begin{prop}\label{prop2}
$\wt(T^*)=\lvert H_{i}(X,\overline{T^*})\rvert$ when $\wt(T^*)\neq 0$. Moreover, we have  $\wt(T^*)=0$ if and only if $\lvert H_{i}(X,\overline{T^*})\rvert =\infty$.
\end{prop}
\begin{proof}
Assume $\wt(T^*)\neq 0$. Then we will show the following equalities:
\begin{eqnarray*}
\wt(T^*)&=&\lvert\det([\delta_{X,i}|_I]_Z)\rvert \\
	&=&\lvert \det([\delta_{(X,\overline{T^*}),i}]_Z)\rvert \\
	&=&\lvert \ker\delta_{(X,\overline{T^*}),i+1} / \im\delta_{(X,\overline{T^*}),i}\rvert \\
	&=&\lvert \WH^{i+1}(X,\overline{T^*})\rvert\\
	&=&\lvert \WH_{i}(X,\overline{T^*})\rvert
\end{eqnarray*}
The second equality comes from \Cref{prop1}. The third equality holds because it is the volume of a lattice. The fourth equality comes from the definition of relative homology. To get the last equality, use \Cref{lem_prop2}. 

Now, if $\wt(T^*)=0$, we have $\det([\delta_{(X,\overline{T^*}),i}]_Z)=0$, and $\lvert \ker\delta_{(X,\overline{T^*}),i+1} / \im\delta_{(X,\overline{T^*}),i}\rvert$ is infinite. Thus, we have $\lvert H_{i}(X,\overline{T^*})\rvert =\infty$. The converse of this statement can be proved similarly.
\end{proof}

We will show that $\overline{T^*}$ is a cycletree previously defined, which establishes a relation between the set of cycletrees $\mathcal{U}_i$ and that of dual spanning trees $\mathcal{T}^i$. For the rest of this paper, we will frequently use the following condition to be applied simultaneously to the (dual) spanning trees, (dual) cycletrees, and  winding(cutting) numbers.

\begin{cond}[Unicycle condition]\label{condition}
Let $X$ be a cell complex with 
\begin{equation}
\rk\WH_{i+1}(X)=0,\ \rk\WH_i(X)=1\textit{, and } \rk\WH_{i-1}(X)=0.
\end{equation}
\end{cond}

\begin{thm}\label{relation1} Let $X$ be a cell complex under \Cref{condition}. The map between $\mathcal{U}_i$ and $\mathcal{T}^i$ given by $U\mapsto\overline{U}=T^*$ is a bijection. Moreover, we have the relation $$\lvert w(C_U)\rvert=\wt(U)\lvert w(\widehat{C_U})\rvert=\wt(T^*)\lvert \WH_{i-1}(X)\rvert.$$ 
\end{thm}

\begin{proof}
For the first statement, we can easily check that the complement operator is bijective. Thus, it is sufficient to show that the image of the complement operator on $\mathcal{U}_i$ is $\mathcal{T}^i$. Recall that we have $C_{i}(X;\mathbb{Q})=\mathcal{H}_{i}(X)\oplus \im\partial_{i}^{t}\oplus\im\partial_{i+1}$ from the Hodge decomposition in \Cref{decompostion}. Let $n=\rk\partial^t_{i}$, $m=\rk\partial_{i+1}$, and $h=\rk\mathcal{H}_{i}(X)$. Then, we have $n=\rk B_{i-1}$ and $m=\rk B^{i+1}$. Due to \Cref{condition}, we have $n=\rk Z_{i-1}$, $m=\rk Z^{i+1} $, and $h=1$. We know that a cycletree $U$ is an element of $\mathcal{S}_i$ with $\lvert I(U) \rvert = n+1$, and a dual spanning tree $T^*$ is an element $\mathcal{S}_i$ with $\lvert I(T^*)\rvert=m$. 

The second statement holds from \Cref{winding} and \Cref{prop2} if $\wt(U)\neq 0$. Now, it is enough to show that $\wt(T^*)=0$ if $\wt(U)=0$ where $U=\overline{T^*}$. Consider the contrapositive, i.e., $\wt(T^*)\neq 0 \Rightarrow \wt(U)\neq 0$. We have a long exact sequence for the pair $(X,U)$, $H_i(X,U)\xrightarrow{\alpha}\WH_{i-1}(U)\xrightarrow{\beta} \WH_{i-1}(X)$. By using $\WH_{i-1}(U) / \ker\beta \cong \im\beta$ and $\im\alpha\cong \ker\beta$, we get $$\rk\WH_{i-1}(U) \le \rk\WH_{i-1}(X) + \rk H_i(X,U)$$
Because $\wt(T^*)\neq 0$ means $\rk H_i(X,U)=0$ by \Cref{prop2}, and $\rk\WH_{i-1}(X)=0$ by \Cref{condition}, we have $\rk\WH_{i-1}(U)=0$, i.e., $\wt(U)\neq 0$.
\end{proof}
Note that we have 
\begin{equation}\label{iff}
\wt(U)=0 \Rightarrow w(C_U)= 0 \Leftrightarrow \wt(T^*)= 0.
\end{equation}
i.e., the map given by $U\mapsto\overline{U}$ is a bijection when it is restricted to the elements with nonzero weights and nonzero winding numbers.




Finally, we have an enumeration formula for the cycletrees as in \Cref{section_cycletree}.
\begin{cor} Let $X$ be a cell complex under \Cref{condition}. Then
$$k^i(X)\lvert \WH_{i-1}(X)\rvert^2 = \sum_{U\in \mathcal{U}_i} w(C_U)^2.$$
\end{cor}

Of course, the left hand side can be easily computed in polynomial time. With \Cref{inner_cor}, we can show 
\begin{cor}\label{cor_harmonic_cycle}
$\lambda\circ\lambda = k^i(X)k_i(X)\lvert \WH_{i-1}(X)\rvert^2.$
\end{cor}

\begin{ex}[Dual spanning tree and cycletree]\label{ex_dualspanningtree}
In Figure 3-A, we have a cell complex $X$. Figure 3-B shows all of the 1-dimensional dual spanning trees in $X$. The first dual spanning tree has weight 0, and the others have weight 1. Figure 3-C shows all of the 1-dimensional cycletrees corresponding to the  dual spanning trees. The cycle parts of cycletrees are marked red. All of these cycletrees has weight 1. However, the winding number of their cycle parts is 1 except for the first cycle part whose winding number is 0.

\begin{figure}[h]
\centering
\includegraphics[angle=0, scale=0.5]{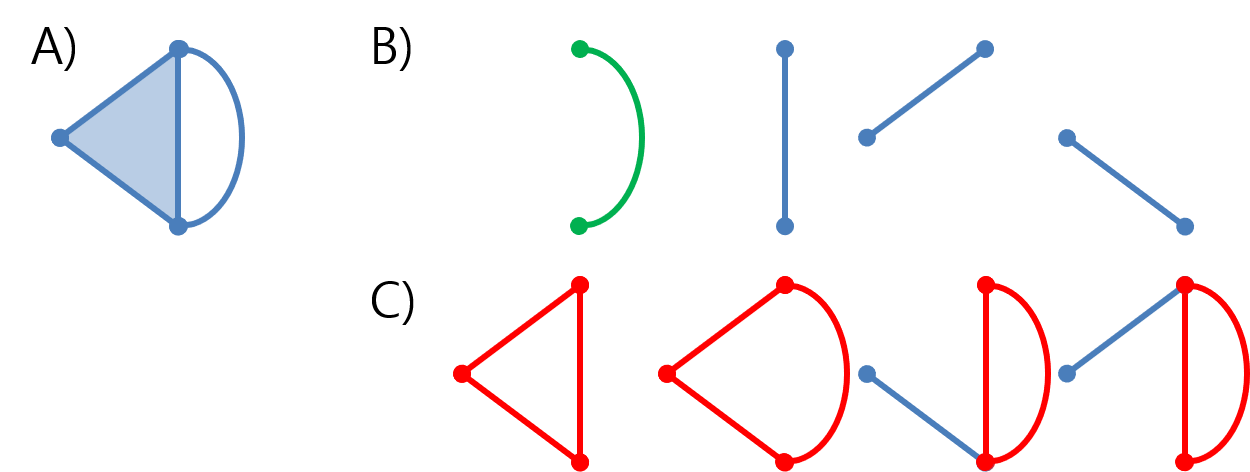}
\caption{A cell complex $X$, dual spanning trees, and cycletrees} \label{Figure3}
\end{figure}

\end{ex}

\section{Dual cycletree and cutting number}\label{section_dual_cycletree}
In this section, we will deal with the dual cycletrees, having one more $i$-cell than the dual spanning trees, and see related properties with respect to the cutting number of a cocycle which is a conceptional dual to the winding number of a cycle. Most of the proofs in this section are similar to those in the previous sections, and will be omitted to avoid repetition.
\subsection{Dual cycletree and its minimal cocycle}

\begin{mydef}
Let $X$ be a cell complex with $\rk\WH^{i+1}(X)=0$. An $i$-dimensional dual cycletree $U^*$ is an element in $\mathcal{S}_i$ such that $\lvert I(U^*)\rvert=\rk Z^{i+1}(X) +1$ with the weight $\wt(U^*)=\lvert H^{i+1}(X,\overline{U^*}) \rvert$ if it is finite, and $\wt(U^*)=0$ otherwise.
\end{mydef}


Let $\mathcal{U}^i=\mathcal{U}^i(X)$ be the set of all $i$-dimensional dual cycletrees in $X$.
When $U^*$ is a dual cycletree with a nonzero weight, we have a unique cocycle $C_{U^*}$ up to nonzero scalar multiplication. Let $C_{U^*}$ be the cocycle constructed by a parallel argument following \Cref{cyclepart}, $\gcd(C_{U^*})$ the gcd of all components of the vector $C_{U^*}$, and $\widehat{C_{U^*}}$ the integer vector $\dfrac{1}{\gcd(C_{U^*})}C_{U^*}$. Note that, if $U^*$ has zero weight, then $C_{U^*}$ and $\widehat{C_{U^*}}$ are the zero vectors.

\begin{prop}\label{prop_gcd_d}
Let $U^*$ be an $i$-dimensional dual cycletree of a cell complex $X$ with the weight $\wt(U^*)$. Then we have $\gcd(C_{U^*})=\lvert H^{i+1}(X,\overline{U^*}) \rvert$. If we assume $\wt(U^*)\neq 0$, then $$\gcd(C_{U^*})=\wt(U^*).$$ Moreover, if $e$ be a virtual $(i-1)$-cell with the coboundary $\delta_{i-1}(e)=m\cdot \widehat{C_{U^*}}$ for a positive integer $m$, and $\rk \WH_i(X)=1$, then we have
$$\lvert H^i(X\oplus e, \overline{U^*})\rvert= m.$$
\end{prop}

For our purpose, duality is closely related to, for example, a graph cut. Given a graph $G$, a graph cut is essentially a binary partition which consists of sources and sinks on the vertex set of $G$. We can represent the cut for the partition as the set of directed edges in $G$ from sources to sinks. A bond is a minimal cut. The cut space of $G$ is a vector space generated by graph cuts. Note that a cut space is identified with $\im\partial^t_1$ and we have $C_1(X)=\ker\partial_1\oplus \im\partial^t_1$. As a high dimensional analog, we define $\im\partial^t_i$ to be the $i$-th cut space. Therefore, the cut space is the orthogonal complement of the cycle space. Refer to \cite{DKM2} for more detail.

Notice that, in a lower dimensional complex $X$ with $\rk \WH_1(X)=0$ or a planar graph $G=X^1$, we have $\im\partial^t_1=\ker\partial^t_2$. Equivalently, the 1-coboundary space $B^1$ and the 1-cocycle space $Z^1$ are the same. A high dimensional cut space is often generalized as a coboundary space $B^i$. This generalization makes sense when we are dealing with an acyclic complex $X$ and focusing on the connectivity of $X$. However, a cocycle space $Z^i$ also has a strong relation with ``cut" in a different sense. 

Refer to \Cref{minimal} for the definition of minimal cocycle. For the proofs of the following propositions, consult with those of \Cref{prop_mincycle1} and \ref{prop_mincycle2}.
\begin{prop}
The cocycle part of a dual cycletree is a minimal cocycle, and vice versa.
\end{prop}

Here, by \Cref{decompostion}, we have $\mathcal{H}_{i}(X)\oplus\im\partial_{i}^{t}=\ker\partial_{i}^{t}$ and know that the set of graph cuts is a subset of the set of cocycles. Moreover, there is a relation between bonds and minimal cocycles.

\begin{prop}
A bond (or, minimal coboundary) can be written as the sum of at most two minimal cocycles.
\end{prop}

\begin{ex}[Minimal coboundary(or, bond) and minimal cocycle]
In Figure 4-A, we have a cell complex $X$. Figure 4-B shows two cluster (or, set of vertices), $c_0$ marked red and $c_1$ marked green. Note that the induced graphs by $c_0$ and $c_1$ are connected. The minimal coboundary marked blue, which is the coboundary of $c_1$, is a minimal cocycle. Figure 4-C shows the red cluster $c_0$ and green cluster $c_1$. The coboundary of $c_1$ is the minimal coboundary, and is the sum of two minimal cocyles $v_0$ (upper cocycle) and $v_1$ (lower cocycle.)

\begin{figure}[h]
\centering
\includegraphics[angle=0, scale=0.4]{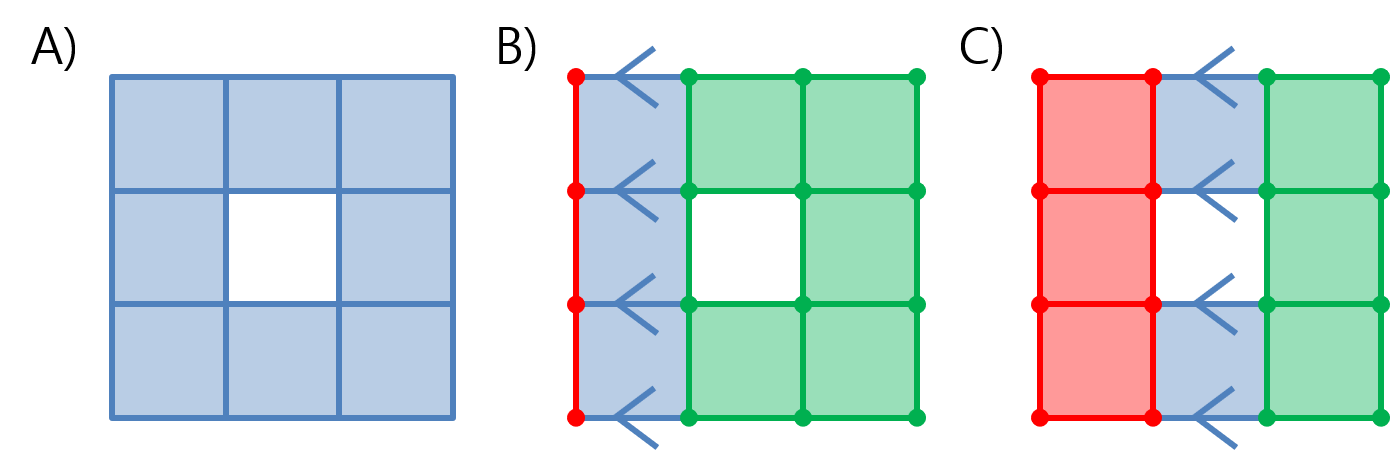}
\caption{A cell complex $X$, minimal cocycles with minimal coboundaries} \label{Figure4}
\end{figure}

\end{ex}

\subsection{Cutting number}
The winding number measures how many times a cycle winds around a homology generator. Similarly, a cutting number counts the multiplicity of cocycles of $\WH^i(X)$ where $X$ is a cell complex with $\rk\WH_i(X)=1$.
\begin{mydef}
A cutting number $c:Z^i\rightarrow \mathbb{Z}$ is given by 
$$c(z)=\det([z]_Z,[\overline{\delta}_{i-1}]_Z)$$ 
where $Z$ is a basis of $Z^i(X)$, and $z$ and the columns of a reduced matrix $\overline{\delta}_{i-1}$ are written with respect to the basis $Z$ to compute the determinant.
\end{mydef}

\begin{ex}[Intuitive meaning of the cutting number]\label{ex_cutting}
As in \Cref{ex_winding}, we have a cell complex $X$ and the standard harmonic cycle $\lambda$ on $X$ in Figure 5-A. In the left hand sides of Figures 5-B and 5-C, the red lines \emph{cut} the complex $X$, which annihilates the 1-dimensional homology of $X$ for B. The right hand sides of Figures 5-B and 5-C show the corresponding red cocyles $v_B$ and $v_C$. Note that $c(v_B)=\pm 1$, and $c(v_C)=0$. We see that the cocycle $v_C$ indeed fails to ``cut" the complex $X$ unlike $v_B$.

\begin{figure}[h]
\centering
\includegraphics[angle=0, scale=0.5]{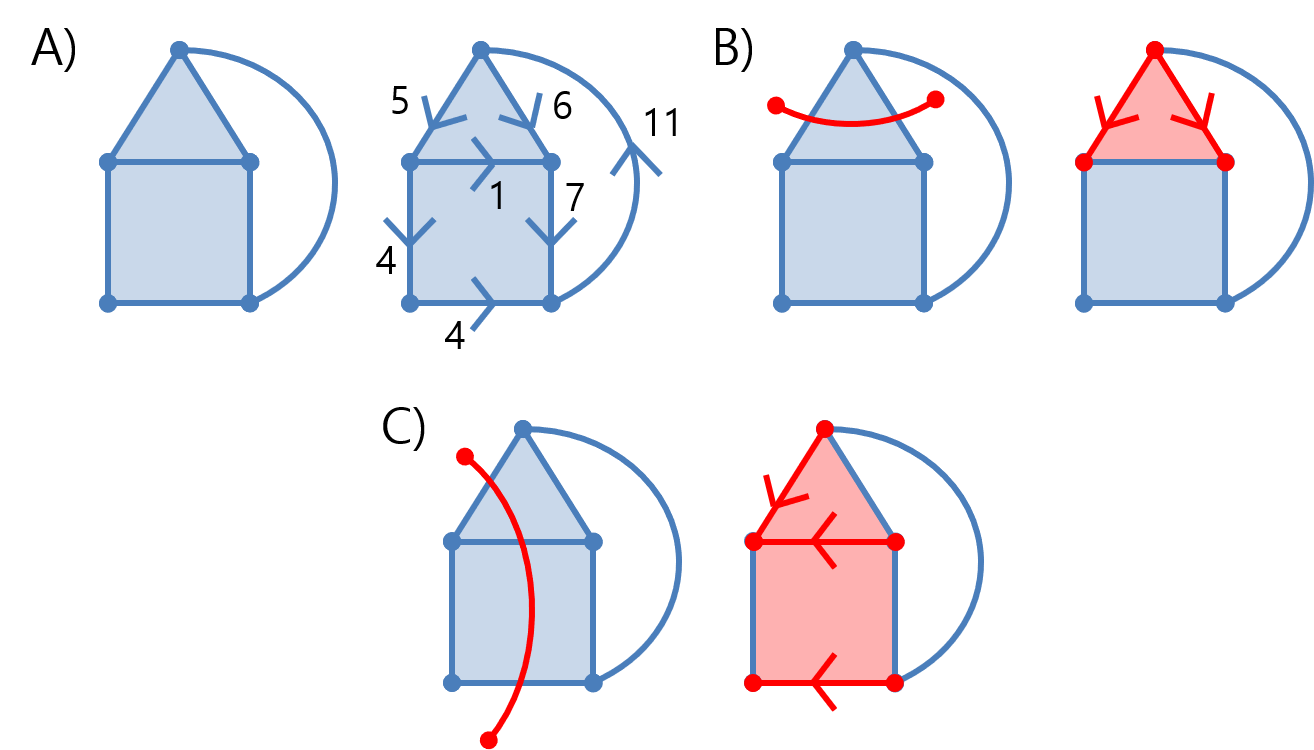}
\caption{A cell complex $X$, a harmonic cycle, and cocycles with corresponding cuts} \label{Figure5}
\end{figure}

\end{ex}

We can represent the cutting number $c(C_{U^*})$ via cohomology as follows.
\begin{prop} \label{prop_cutting}
For a cocycle $z\in Z^{i}$, let $e$ be a virtual $(i-1)$-cell with its coboundary $\delta_{i-1}(e) = z$. If $c(z)\neq 0$, then $$\lvert c(z)\rvert  = \lvert\WH^i(X\oplus e)\rvert.$$
Otherwise, $c(z)\neq 0$ if and only if $\lvert\WH^i(X\oplus e)\rvert=\infty$.
\end{prop}


Like the relation \Cref{relation1} between dual spanning trees and cycletrees, there is a relation between dual cycletrees and spanning trees.


\begin{prop} \label{prop_cutting2}
Let $U^*$ be a dual cycletree with $c(C_{U^*})\neq 0$. Then, $\wt(U^*) \lvert c(\widehat{C_{U^*}})\rvert=\lvert\WH^i(\overline{U^*})\rvert\lvert\WH^{i+1}(X)\rvert$.
\end{prop}
\begin{proof}
Since $c(C_{U^*})\neq 0$, we have $\wt(U^*)\neq 0$.
Let $e$ be a virtual $(i-1)$-cell whose coboundary is $\widehat{C_{U^*}}$. Then, by \Cref{prop_gcd_d}, we have $H^i(X\oplus e,\overline{U^*})=0$ because $\lvert H^i(X\oplus e, \overline{U^*})\rvert=1$. Moreover, it is easy to show that $\WH^{i+1}(\overline{U^*})=0$.

Now, there is a long exact sequence of cohomology groups for the pair of chain complexes $(C_i(X\oplus e), C_i(\overline{U^*}))$
$$ 0 \to \WH^i(X\oplus e) \to \WH^i(\overline{U^*}) \to H^{i+1}(X\oplus e ,\overline{U^*}) \to \WH^{i+1}(X\oplus e) \to 0. $$
%
%
%
%
%
Since $e$ is a virtual $(i-1)$-cell, we know that $\WH^{i+1}(X\oplus e)=\WH^{i+1}(X)$, and $H^{i+1}(X\oplus e,\overline{U^*})=H^{i+1}(X,\overline{U^*})$. Hence, the above long exact sequence is
$$ 0 \to \WH^i(X\oplus e) \to \WH^i(\overline{U^*}) \to H^{i+1}(X,\overline{U^*}) \to \WH^{i+1}(X) \to 0. $$
Therefore, $\lvert\WH^i(X\oplus e)\rvert\lvert H^{i+1}(X,\overline{U^*})\rvert=\lvert\WH^i(\overline{U^*})\rvert\lvert\WH^{i+1}(X)\rvert$ holds. By \Cref{prop_cutting}, we have $\lvert\WH^i(X\oplus e)\rvert=\lvert c(\widehat{C_{U^*}})\rvert$, and the proposition holds.

\end{proof}

\begin{thm}\label{relation2}
Let $X$ be a cell complex under \Cref{condition}. There is a bijection between $\mathcal{U}^i$ and $\mathcal{T}_i$ via the complement operator. Moreover, the bijection between $U^*\in \mathcal{U}^i$ and $T\in \mathcal{T}_i$ preserves $$\lvert c(C_{U^*})\rvert=\wt(U^*)\lvert c(\widehat{C_{U^*}})\rvert=\wt(T)\lvert\WH^{i+1}(X)\rvert$$ where $U^*=\overline{T}$.
Also, we have $c(C_{U^*})=0$ if and only if $\wt(T)=0$.
\end{thm}
\begin{proof}
The proof is similar to that of \Cref{relation1} combined with \Cref{prop_cutting2}.
\end{proof}

\begin{cor}\label{cor_weighted_cycletree_number_d1}
The $i$-tree number $k_i(X)$ can be given as $$k_i(X)= \dfrac{1}{\lvert\WH^{i+1}(X)\rvert^2}\sum_{U^*\in \mathcal{U}^i} c(C_{U^*})^2.$$
\end{cor}
\begin{ex}[Spanning trees and dual cycletrees]
As in \Cref{ex_dualspanningtree}, we have a cell complex $X$ in Figure 6-A. Figure 6-B  shows 1-dimensional spanning trees in $X$. The last spanning tree has weight 0 while the others have weight 1. Figure 6-C shows 1-dimensional dual cycletrees in $X$ which correspond to the spanning trees. The cocycle parts of cycletrees are colored red. Every cycletree has weight 1. However, the cutting numbers of their cocycle parts are 1 except for the last. The last cocycle part has a cutting number of 0.

\begin{figure}[h]
\centering
\includegraphics[angle=0, scale=0.5]{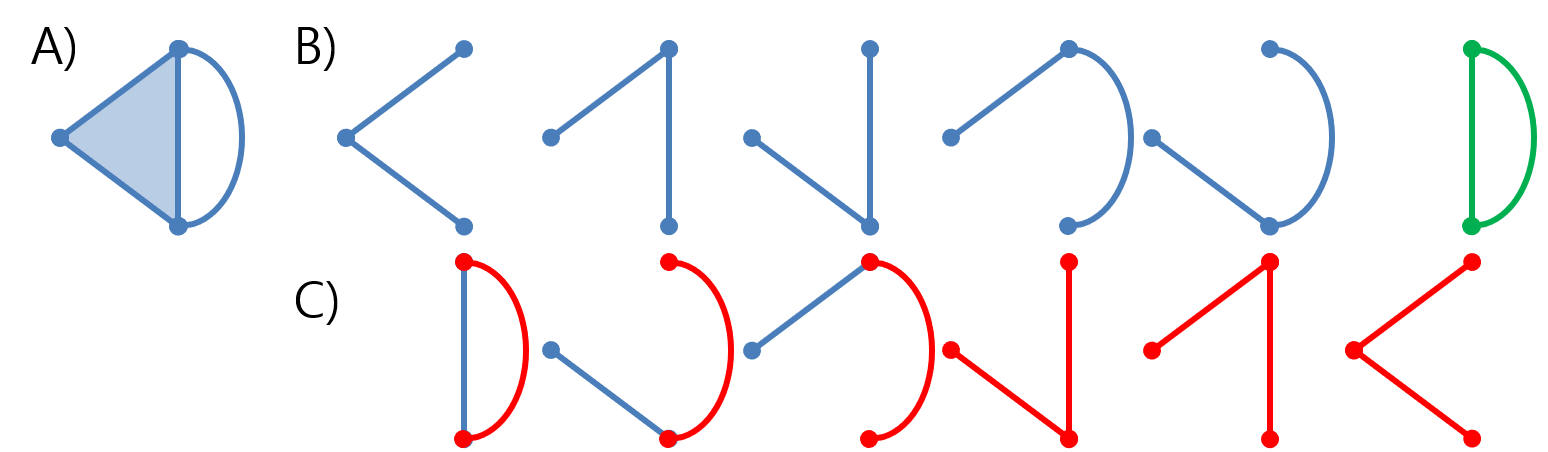}
\caption{A cell complex $X$, spanning trees, and dual cycletrees} \label{Figure6}
\end{figure}

\end{ex}

\section{Standard harmonic cocycle and relationship}\label{section_harmonic_cocycle}
In this section, we will see a decomposition of harmonic cycle with respect to cocycles. And this is the dual statement of \Cref{hc}. Some proofs which are already mentioned in dual statements are omitted.

\begin{mydef}
Let $X$ be a cell complex under the conditions (\Cref{condition}). The standard harmonic cocycle $\lambda^*$ in $Z^i(X)$ is given by
$$\lambda^* = \sum_{U^*} c(C_{U^*}) C_{U^*}.$$
where the summation is over all $U^*\in \mathcal{U}^i(X)$.
\end{mydef}

This expression is made by replacing cycles with cocycles. The standard harmonic cocycle has a relation with the cutting number that is similar to the relation in \Cref{inner}. 
\begin{thm}\label{inner2}
Let $X$ be a cell complex under \Cref{condition}. We have $$\lambda^*\circ x = c(x)\cdot k^i(X).$$
\end{thm}

Hence, $\lambda^*$ is seen to be non-trivial as in \Cref{inner_cor}. 
\begin{cor}
$\lambda^*$ is a nonzero element in $\mathscr{H}_i(X)$.
\end{cor}

\begin{ex}[Inner product formula]
In \Cref{ex_cutting}, there are the standard harmonic cocycle $\lambda$ of a cell complex $X$, and cocycles $v_B$, and $v_C$. From the inner product formula for $v_B$, we have $v_B\circ \lambda =5+6=11 = 1\times 11 = c(v_B)\cdot k^1(X)$ where $k^1(X)=11$. For $v_C$, we have $v_C\circ \lambda =5-1-4=0 = 1\times 0 = c(v_C)\cdot k^1(X)$.
\end{ex}

\begin{remark}
As in the rational winding number, via \ref{inner2}, we can extend the domain of cutting number map to a chain group. Specifically, we define the rational cutting number to be  $c(z) \coloneqq z\circ \lambda / k^i(X)$. 
\end{remark}

From the standard harmonic cocycle, we can calculate the sum of the weighted dual cycletrees.
\begin{cor}\label{cor_weighted_cycletree_number_d2}
$$\dfrac{\lambda^* \circ \lambda^*}{k^i(X)} = \sum_{U^*\in \mathcal{U}^i} c(C_{U^*})^2.$$ 
\end{cor}

By using \Cref{cor_weighted_cycletree_number_d1} and \Cref{cor_weighted_cycletree_number_d2}, we have the following formula.
\begin{cor}\label{cor_harmonic_cocycle}
$$\lambda^* \circ \lambda^* = k^i(X) k_i(X)\lvert\WH^{i+1}(X)\rvert^2.$$ 
\end{cor}

Further, we point out the following elegant relation between the standard harmonic cycle and the standard harmonic cocycle.
\begin{thm}\label{thm_normal_def}
$$\dfrac{\lambda}{\lvert\WH_{i-1}(X)\rvert} =\pm \dfrac{\lambda^*}{\lvert\WH^{i+1}(X)\rvert}.$$
\end{thm}
\begin{proof}
Use \Cref{cor_harmonic_cycle} and \Cref{cor_harmonic_cocycle}.
\end{proof}

Finally, we present the normalized harmonic (co)cycle as follows.

\begin{thm}
Let the normalized harmonic cycle $\overline{\lambda}$ be $\dfrac{\lambda}{\lvert\WH_{i-1}(X)\rvert}$. Then,  $$\overline{\lambda}\circ\overline{\lambda}= k_i(X)k^i(X).$$
\end{thm}
\begin{proof}
Use \Cref{cor_harmonic_cocycle} and \Cref{thm_normal_def}.
\end{proof}

It is worth noting that a similar form can be seen in \cite{Ca2, L}. However, the result is usually focused on the case when $\rk\WH_i(X)=0$, whereas we have $\rk\WH_i(X)=1$ under \Cref{condition}.

\begin{remark}
We may summarize the relationships among the previous topological and combinatorial objects in the following diagram.
\begin{center}
\begin{tikzpicture}
  \matrix (m) [matrix of math nodes,row sep=3em,column sep=4em,minimum width=2em]
  {
     \mathcal{T}_i & \mathcal{U}^i  \\
     \mathcal{U}_i & \mathcal{T}^i  \\};
  \path[-stealth]
    (m-1-1) edge [thin] node [left] {$add$} (m-2-1)
    (m-2-2) edge [thin] node [right] {$add$} (m-1-2)
    (m-1-1) edge [dashed,-] (m-2-2)
    (m-1-2) edge [dashed,-] (m-2-1);
  \draw[<->] 
    (m-1-1) edge node [above] {$complement$} (m-1-2)
    (m-2-1) edge node [below] {$complement$}(m-2-2);
\end{tikzpicture}
\end{center}
The dotted lines represent conceptual duality between objects, which appears when we use $\delta$ instead of $\partial$. The arrows with `add' are surjective maps, and they correspond to adding exactly one cell to a (dual) spanning tree. Finally, `complement' means the complement operator on $\mathcal{S}_i$.
\end{remark}






\section*{Acknowledgment} Y.-J. Kim was supported by NRF(National Research Foundation of Korea) Grant funded by Korean Government (NRF-2015-Global Ph.D. Fellowship Program). W. Kook was supported by the National Research Foundation of Korea (NRF) Grant funded by the Korean Government (MSIP) (No.2017R1A5A1015626 and 2018R1A2A3075511).

\end{document}